\documentclass {article}
\usepackage{authblk}
\usepackage{etex}
\usepackage[utf8]{inputenc}
\usepackage[english]{babel}
\usepackage{amsmath}
\usepackage{amsthm}
\usepackage{amssymb}
\usepackage{sectsty}
\usepackage{titlesec}
\usepackage{color}
\usepackage[color,matrix,arrow]{xy}
\usepackage{amsgen}
\usepackage{amstext}
\usepackage{amsbsy}
\usepackage{amsopn}
\usepackage{amsfonts}
\usepackage{eepic}
\usepackage{graphicx}
\usepackage{epsf}
\usepackage{pstricks}
\usepackage{enumerate}
\usepackage{eqnarray}
\usepackage{faktor}
\xyoption{all}

\usepackage{pgf}
\usepackage{tikz-cd}
\usetikzlibrary{automata, arrows.meta, positioning,decorations.pathmorphing}

\newcommand{\footrecall}[1]{%
} 
\usetikzlibrary{snakes,shapes,arrows,automata}

\titleformat*{\section}{\large\bfseries}
\titleformat*{\subsection}{\normalsize \bfseries}

\newcommand{\Cay}{\operatorname{Cay}}
\newcommand{\Geo}{\operatorname{Geo}}

\theoremstyle{definition}
\newtheorem{theorem}{Theorem}[section]
\newtheorem{corollary}[theorem]{Corollary}

\newtheorem{proposition}[theorem]{Proposition}
\newtheorem{question}[theorem]{Question}

\newtheorem{conjecture}[theorem]{Conjecture}
\newtheorem{lemma}[theorem]{Lemma}

\begin{document}

\title{A  non-plain virtually free group with a strictly $2$-geodetic Cayley graph}
\author{Andr\'e Carvalho\\
	\small Research Center in Mathematics and Applications (CIMA)\\
	\small Department of Mathematics, School of Sciences and Technology, University of \'Evora\\
	\small Rua Rom\~ao Ramalho, 59, 7000--671 \'Evora, Portugal\\
	\small \texttt{andrecruzcarvalho@gmail.com}}
\date{}
\maketitle
\begin{abstract}
	We construct an explicit infinite virtually free group
	$G=C_4\ast_{C_2}D_3$
	with a finite inverse-closed generating set whose Cayley graph is $2$-geodetic but not geodetic. More precisely, we describe the full language of geodesics and show that every element of $G$ has a unique geodesic representative except for a single element, which has exactly two. This gives a counterexample to the conjecture that an infinite group admitting a $k$-geodetic Cayley graph must admit a geodetic Cayley graph with respect to the same generating set. We further prove that $G$ is freely indecomposable and hence non-plain. Consequently, Shapiro's conjecture that every geodetic group is plain and the conjecture that every hyperbolic $k$-geodetic group is geodetic cannot both hold.
\end{abstract}
\section{Introduction}

A connected graph is called $k$-geodetic if every pair of vertices is
joined by at most $k$ geodesics and it is strictly $k$-geodetic but not $(k-1)$-geodetic.  A group is called $k$-geodetic if
it has a finite inverse-closed generating set whose Cayley graph is
$k$-geodetic. A group is said to be geodetic if it is $1$-geodetic.

A group is said to be plain if it is a free product of finitely many
finite groups and finitely many infinite cyclic groups. Shapiro asked whether the geodetic groups are exactly the plain groups
\cite{Shapiro}.  Madlener and Otto  conjectured that the groups that admit presentation by finite confluent length-reducing rewriting systems are exactly the plain groups \cite{MO}.

Papasoglu proved that hyperbolic geodetic groups must be virtually free \cite{Papasoglu}, but it is not known if all geodetic groups are hyperbolic nor if virtually free geodetic groups must be plain.
Elder, Piggott and Townsend proved an analogue to Papasoglu's result, that is, that every hyperbolic $k$-geodetic
group is virtually free \cite[Theorem B]{EPT}.  They also made the
following conjecture \cite[Conjecture C]{EPT}:

\begin{conjecture}\label{conjectureelderetal}
	A hyperbolic group is $k$-geodetic for some $k$ if and only if
	it is geodetic.  Furthermore, if $G$ is infinite and
	$\Cay(G,S)$ is $k$-geodetic for some finite generating set $S$,
	then $\Cay(G,S)$ is geodetic.
\end{conjecture}

 We give a counterexample to the second statement. Concretely, we show that 
 	$$
 G=\left\langle r,t\ \middle|\ r^4=t^2=(r^2t)^3=1\right\rangle\cong C_4\ast_{C_2} D_3
 $$
 is infinite, virtually free, and, if $
 S=\{r,r^{-1},r^2,t\}, 
 $
  then
 $
 \Cay(G,S)
 $
 is $2$-geodetic but not geodetic. More than that, there is only one geodesic between $1$ and $g$ for every $g\neq r^2tr^2$ and the element $r^2tr^2$ has exactly two geodesics. 
 Concretely, let
 $
 P=\{r,r^{-1},s\}=A\setminus\{1\},
 $
 and $L_0\subseteq X^*$ be the language consisting of the empty word together with the words $u$ satisfying:
 \begin{enumerate}[(i)]
 	\item $u$ alternates between letters of $P$ and the letter $t$;
 	\item  $s$ only appears as the first or last letter of $u$.
 \end{enumerate}
 
 We show that this language is an automatic structure with uniqueness for $G$ and that $L=L_0\cup \{tst\}$ is precisely the language of geodesics of $G$ with respect to $X$.
 
 \newtheorem*{main}{Theorem \ref{prop:D3-injective}}
 \begin{main}
 	$L_0$ is a geodesic automatic structure with uniqueness for $G$. Consequently, $\Geo_X(G)=L$. 
 \end{main}
 
 As a corollary, we have that this Cayley graph is $2$-geodetic but not geodetic, providing a counterexample to the second statement of Conjecture \ref{conjectureelderetal}.
 
  \newtheorem*{maincor}{Corollary \ref{maincor}}
 \begin{maincor}
 	The Cayley graph $\Cay(G,X)$ is 2-geodetic but not geodetic.  The only element
 	with two geodesic representatives is $sts=tst$.
 \end{maincor}

 Interestingly, we also show that this group is not plain.  We do not know if there is some generating set $X$ such that $\Cay(G,X)$ is geodetic. 
 If there is, $G$ is a counterexample to Shapiro's question; if there is not, then $G$ is also counterexample to the first statement of Conjecture \ref{conjectureelderetal}.

 \newtheorem*{either}{Corollary \ref{cor:either}}
 \begin{either}
 	At least one of the following is false:
 	\begin{enumerate}
 		\item every geodetic group is plain;
 		\item every hyperbolic $k$-geodetic group is geodetic.
 	\end{enumerate}
 	Thus Shapiro's conjecture and the first assertion of Conjecture \ref{conjectureelderetal} cannot
 	both be true.
 \end{either}

Notice that $\Cay(G,X)$ is very close to being geodetic, in the sense that only one element has two geodesics, while all others have a single one. However, we show that it is not possible to add generators to our set and make the Cayley graph become geodetic.

 \newtheorem*{enlarge}{Proposition \ref{prop:no-geodetic-enlargement}}
\begin{enlarge}
	Let
	$
	S=\{r,r^{-1},t,s\}.
	$
	Suppose that $X$ is an inverse-closed generating set containing $S$
	and that $\operatorname{Cay}(G,X)$ is geodetic. Then
	$
	G\setminus\{1\}\subseteq X.
	$
	In particular, no finite inverse-closed generating set containing
	$S$ gives a geodetic Cayley graph.
\end{enlarge}

We end the paper with several open questions.

\section{Free products and normal forms}

Let $A$ and $B$ be groups with a common subgroup $C$.
The \emph{free product of $A$ and $B$ amalgamating $C$}, denoted
$
A *_C B,
$
is the quotient of the free product $A*B$ obtained by identifying the
two copies of $C$.

An expression
$
x_1x_2\cdots x_n
$
is called \emph{reduced} if
$$
x_i \in (A\setminus C)\cup(B\setminus C)
$$
for every $i$, and consecutive $x_i$ belong to different factors.

Choose left transversals $T_A$ and $T_B$ for $C$ in $A$ and $B$,
respectively, both containing $1$. By the normal form theorem for free
products with amalgamation \cite[Chapter IV, Theorem 2.6]{LS}, every element $g\in A*_C B$ has a unique
expression
$$
g=u_1u_2\cdots u_n c,
$$
where $c\in C$, 
$
u_i\in (T_A\setminus\{1\})\cup(T_B\setminus\{1\}),
$
and consecutive $u_i$ belong to different factors. We call
$u_1,\ldots,u_n$ the \emph{syllables} of the normal form. When $n=0$,
the normal form means simply $g=c\in C$.

We will also use the following technical lemmas, which follow from the normal forms.

\begin{lemma}\label{lem:comparison}
	Let $G=A\mathbin{*}_H B$.  Suppose
	$$
	x_1\cdots x_m
	\qquad\text{and}\qquad
	y_1\cdots y_\ell
	$$
	are reduced expressions with the same value in $G$.  Then $m=\ell$, and the
	$i$th syllables $x_i,y_i$ belong to the same free factor for every $i$.  Moreover,
	there are $c_1,\ldots,c_{m-1}\in H$ such that
	\begin{align*}
		y_1&=x_1c_1,\\
		y_i&=c_{i-1}^{-1}x_ic_i\qquad(2\le i\le m-1),\\
		y_m&=c_{m-1}^{-1}x_m.
	\end{align*}
	When $m=1$, equality simply means $x_1=y_1$.
\end{lemma}

\begin{proof}
	Choose left transversals $T_A,T_B$ for $H$ in $A,B$, each containing
	$1$.   We will convert both reduced expressions to one in normal form. Let $r_1\in T_A\cup T_B$ and $h_1\in H$ be such that $x_1=r_1h_1$. Then, for every $2\leq i\leq m$, let $r_i\in T_A\cup T_B$ and $h_i\in H$ be such that $h_{i-1}x_i=r_ih_i$.
	
	Notice that, since $x_1\cdots x_m$ is reduced, then $r_1\cdots r_mh_m$ is in normal form because $x_i$ lies in the same free factor as $r_i$ and $r_i\neq 1$ for all $i\in [m]$. Also,  $$x_1\cdots x_m=r_1\cdots r_mh_m.$$ Proceeding analogously, writing $s_i$ instead of $r_i$ and $h_i'$ instead of $h_i$, we have that 
	$$y_1\ldots y_\ell=s_1\cdots s_\ell h_\ell',$$ for $s_i\in T_A\cup T_B$ and $h_\ell'\in H$.
	
	Since both expressions on the right hand side are in normal form, we have that $m=\ell$, $s_i=r_i$ for all $i\in [m]$ and $h_m=h_\ell'$.

	For $1\leq i\leq m-1$, let
	$
	c_i=h_i^{-1}h_i'\in H.
	$
	Then
	$$
	y_1=r_1h_1'=r_1h_1(h_1^{-1}h_1')=x_1c_1.
	$$
	For every $2\leq i\leq m-1$, we have
	$$
	x_i=h_{i-1}^{-1}r_ih_i
	\qquad\text{and}\qquad
	y_i=(h_{i-1}')^{-1}r_ih_i'.
	$$
	Therefore,
	\begin{align*}
		c_{i-1}^{-1}x_ic_i
		&=(h_{i-1}^{-1}h_{i-1}')^{-1}
		(h_{i-1}^{-1}r_ih_i)
		(h_i^{-1}h_i')\\
		&=(h_{i-1}')^{-1}r_ih_i'\\
		&=y_i.
	\end{align*}
	Finally, since $h_m=h_m'$, we have
	\begin{align*}
		c_{m-1}^{-1}x_m
		&=(h_{m-1}^{-1}h_{m-1}')^{-1}
		(h_{m-1}^{-1}r_mh_m)\\
		&=(h_{m-1}')^{-1}r_mh_m'\\
		&=y_m.
	\end{align*}
	\end{proof}

We will also use the fact that conjugates of free factors are either equal or intersect trivially \cite[Lemma 2.1]{HandelMosher}. This follows easily from Bass--Serre theory. Indeed, in the Bass–Serre tree of $H*K$, conjugates of $H$ and $K$ are vertex stabilizers, and edge stabilizers are trivial. If two distinct vertex stabilizers had a nontrivial element in common, that element would fix both vertices and therefore the geodesic joining them. In particular, it would fix an edge, contradicting the triviality of edge stabilizers. 

\begin{lemma}\label{lem:free-factor-intersections}\cite[Lemma 2.1]{HandelMosher}
	Let
	$
	P=\mathop{\ast}_{\lambda\in\Lambda} P_\lambda
	$
	be a free product. For all $x,y\in P$ and
	$\lambda,\mu\in\Lambda$, either
	$$
	xP_\lambda x^{-1}=yP_\mu y^{-1},
	\qquad
	\text{ or } \qquad
	xP_\lambda x^{-1}\cap yP_\mu y^{-1}=\{1\}.
	$$
\end{lemma}

We also recall the Kurosh Subgroup Theorem, which describes the structure of
subgroups of free products.

\begin{theorem}[Kurosh Subgroup Theorem]
	Let
	$
	G = \ast_{\lambda \in \Lambda} G_\lambda
	$
	be a free product, and let $H \leq G$. 
	 Then $H$ is the free product of a free group together
	with groups that are conjugates of subgroups of the free factors $G_\lambda$ of $G.$\end{theorem}
\section{A virtually free group with a strictly 2-geodetic Cayley graph}
We will now construct a non-plain virtually free group with a   strictly $2$-geodetic Cayley graph, answering to the second half of Conjecture \ref{conjectureelderetal} in the negative. To do so, we will describe the language of geodesics of $G$ with respect to $X$. We start by introducing some definitions.

Let $G$ be a group generated by a finite set $A$, and let
$\pi\colon A^* \to G$ denote the natural evaluation map. A
\emph{cross-section} for $G$ over $A$ is a language $L \subseteq A^*$
such that the restriction
$
\pi|_L\colon L \to G
$
is a bijection. Thus, every element of $G$ is represented by exactly
one word in $L$. If $L$ is a rational language, then we call $L$ a
\emph{rational cross-section}.

An \emph{automatic structure} for $G$ consists of a finite generating
set $A$ and a rational language $L \subseteq A^*$ mapping surjectively
onto $G$, for which there exists a constant $K \geq 0$ such that
whenever $u,v\in L$ satisfy
\[
d_G({u}\pi,{v}\pi)\leq 1,
\]
the paths in the Cayley graph $\Gamma(G,A)$ labelled by $u$ and $v$,
and starting at the identity, synchronously $K$-fellow travel. That is,
if $u(t)$ and $v(t)$ denote the points reached at time $t$ along these
paths (with a path remaining at its endpoint after it terminates), then
\[
d_G\bigl(u(t),v(t)\bigr)\leq K
\]
for all $t\geq 0$. A group admitting such a structure is called an
\emph{automatic group}. If $L$ is moreover a cross-section, then the
automatic structure is said to have uniqueness.

  Put
$$
A=\langle r\mid r^4=1\rangle,
\qquad
B=\langle s,t\mid s^2=t^2=(st)^3=1\rangle\cong D_3,
$$
and amalgamate
$
C=\langle s\rangle=\{1,s\}
$
with $s=r^2\in A$. Let
$X=\{r,r^{-1},s,t\}$ be a generating set for $G=A\ast_C B$. Notice that $X$ is inverse-closed as $s,t$ are involutions.

The group $G$ is infinite. Indeed, for every $m\geq1$ the word
$
(rt)^m
$
is a reduced normal form, and these normal forms are pairwise
distinct.

Moreover, $G$ is virtually free. This follows from the theorem of
Karrass, Pietrowski and Solitar~\cite{KPS}, since $G$ is the
fundamental group of a finite graph of finite groups with vertex
groups $C_4$ and $S_3$ and edge group $C_2$. In particular, $G$
is hyperbolic.

Notice that  $B=\{1,s,t,st,ts,sts\}$, and that $sts=tst$. Indeed $(st)^2=(st)^{-1}=ts$, so $stst=ts$ and multiplying by $s$ on the left yields $tst=sts$. 

Let
$
P=\{r,r^{-1},s\}=A\setminus\{1\}
$
and $L_0\subseteq X^*$ be the language consisting of the empty word together with the words $u$ satisfying:
\begin{enumerate}[(i)]
	\item $u$ alternates between letters of $P$ and the letter $t$;
	\item  $s$ only appears as the first or last letter of $u$.
\end{enumerate}
  Finally, we define 
$$L=L_0\cup\{tst\}.$$

We will show that $L=\Geo_X(G)$. We first show an inclusion.
\begin{lemma}\label{lem:D3-geodesics-in-L}
	Every geodesic word over $X$ belongs to $L$.
\end{lemma}

\begin{proof}
	Let $w$ be geodesic. Clearly, it cannot contain two consecutive letters from $P$, as their product is either empty or represented by a single letter in $P$. The word cannot contain
	$tt$ either, since $t^2=1$.  Thus $w$ alternates between $P$ and $t$.
	
	Suppose an occurrence of $s$ is not at an endpoint. Since letters from $P$ alternate with $t$, there is a factor of the form $tst$. If $w=tst$, then $w\in L$ and we are done. If not, then it has an adjacent letter $p\in P$ on at least one side.  Suppose it is on the left, that is, that $ptst$ is a facor of $w$ for some $p\in P$. Then,  since $tst=sts$, 
	$$
	ptst=psts=(ps)ts.
	$$
	But $ps\in A$ is either $1$ or a letter of $P$ and $ptst$ is not geodesic. If, on the other hand, $p$ is on the right, that is $tstp$ is a factor of $w$, then 
	
	$$
	tstp=stsp=st(sp),
	$$
	and again $w$ is not geodesic.   Therefore, $s$ can only occur as either the first or last letter of $w$ and $w\in L$.
\end{proof}

Let $w\in L_0$ contain at least one occurrence of $t$. We will view $w$
as a reduced expression in $A*_C B$ as follows. If $w$ starts with $s$,
then, by the definition of $L_0$, this $s$ is followed by $t$, and we
group the initial subword $st$ as one element of $B\setminus C$. Similarly,
if $w$ ends with $s$, we group the terminal subword $ts$ as one element
of $B\setminus C$. All remaining letters $r,r^{-1}$ and $t$ are left
unchanged. Thus, for example,
$$
strts=(st)\,r\,(ts).
$$
If $w=sts$, we regard the whole word as the single element $sts\in
B\setminus C$. In this way, every word in $L_0$ containing $t$ is viewed
as a reduced expression.

\begin{theorem}\label{prop:D3-injective}
$L_0$ is a geodesic automatic structure with uniqueness for $G$. Consequently, $\Geo_X(G)=L$. 
\end{theorem}
\begin{proof}
	We start by proving that $L_0$ is rational. Let
	$$
	Q=\{r,r^{-1}\}
	\qquad\text{and}\qquad
	P=\{r,r^{-1},s\}.
	$$
	Recall that $L_0$ consists of the words which alternate between letters
	of $P$ and the letter $t$, with $s$ allowed only as the first or last
	letter of the word.
	
	We claim that
	\begin{align*}
		L_0=\{\varepsilon\}\cup P
		\cup P(tQ)^*tP
		\cup P(tQ)^*t		\cup t(Qt)^*P
		\cup t(Qt)^*.
	\end{align*}
	Indeed, the first two languages are the words of $L_0$ containing no $t$. If a
	word contains $t$, there are four possibilities according to whether its
	first and last letters belong to $P$ or are equal to $t$. Since every
	letter from $P$ which is not an endpoint must belong to $Q$, those are precisely
	$$
	P(tQ)^*tP,\qquad
	P(tQ)^*t,\qquad
	t(Qt)^*P,\qquad
	t(Qt)^*.
	$$
 Hence $L_0$ is rational.

	We will now show that it is a cross-section.  By Lemma~\ref{lem:D3-geodesics-in-L}, every geodesic word belongs to $L$. Notice that $tst=sts$ and $sts\in L_0$. So, $L_0$ contains a representative for every element of $G$, that is, the evaluation map is surjective. 
	
	We will now show that it is injective. 
	Let $u,v\in L_0$ represent the same element of $G$. If neither word
	contains $t$, then each belongs to
	$
	\{1,s,r,r^{-1}\},
	$
	and the claim is immediate.

Suppose that exactly one of $u,v$ contains $t$, say $v$. Then $u$
represents an element of $A$. If $v$ does not start or end with $s$, then
$v$ is reduced. If $v$ starts with $s$, then this occurrence of $s$ is
followed by a $t$, and we group these two letters as the element $st\in
B\setminus C$. Similarly, if $v$ ends with $s$, we group the final two
letters as the element $ts\in B\setminus C$. If $v=sts$, we regard the
whole word as the element $sts\in B\setminus C$. In every case, this gives
a reduced expression for $v$ containing at least one term from
$B\setminus C$. Since $u$ is a reduced expression with one term from
$A\setminus C$, by Lemma~\ref{lem:comparison}, we have that $u$ and $v$ must represent distinct elements.

	So, suppose that both $u$ and $v$ contain $t$.
	Again, we will see $u$ and $v$ as reduced   expressions by grouping an initial
	$s$ with the following $t$, and a terminal $s$ with the preceding $t$.
	All letters $r,r^{-1}$ are terms from $A\setminus C$, and the terms from
	$B\setminus C$ are obtained from the letters $t$ together with a possible
	endpoint $s$.
	By Lemma~\ref{lem:comparison}, the two reduced expressions have the same
	number of terms, corresponding terms belong to the same factors, and
	there are elements $c_i\in C=\{1,s\}$ relating corresponding terms as in
	that lemma.
	
	We claim that every $c_i$ is trivial, and so that $u=v$. Consider first a $B$-term which is
	not the first or last term of the reduced expression. Since $s$ can occur
	only at an endpoint of a word in $L_0$, this term is $t$ in both $u$ and
	$v$. Hence the comparison lemma gives
	$$
	t=c^{-1}td
	$$
	for some $c,d\in\{1,s\}.
	$
	The four possibilities for the right hand side are
	$$
	t,\qquad st,\qquad ts,\qquad sts,
	$$
	which are distinct elements of $B$. Therefore $c=d=1$.

	Now suppose that the first term of the reduced expression belongs to $B$.
	It is either $t$ or $st$ in each word. If $x$ and $y$ are the corresponding first
	terms of $u$ and $v$, then
	$
	y=xc
	$
	for some $c\in\{1,s\}.
	$
	Suppose that $c=s$. Then, if $x=t$, we get $y=ts$ and if $x=st$, we get $y=sts$. But $y$ must be wither $t$ or $st$, which is absurd.
	Hence $c=1$. Similarly, if the last term belongs to $B$, then it is
	either $t$ or $ts$, and the equation
	$
	y=c^{-1}x
	$
	forces $c=1$, since
	$$
	st\notin\{t,ts\}
	\qquad\text{and}\qquad
	sts\notin\{t,ts\}.
	$$

Thus, in every
	case all the elements $c_i$ in a $B$-term are trivial.
	Since the factors alternate, every $c_i$ is adjacent to a $B$-term, and so 
	$
	c_i=1
	$
	for every $i$ and $u=v$. Therefore, $L_0$ is a rational cross-section for $G$ with respect to $X$. 
	
	Now, by Lemma~\ref{lem:D3-geodesics-in-L}, every geodesic word belongs to $L$. Notice that $L=L_0\cup\{tst\}$ and $tst=sts$. Also $sts\in L_0$. So $L_0$ contains a geodesic for every element of $G$. Since $L_0$ is a rational cross-section, its representative of each element is unique. It follows
	that every word in $L_0$ is geodesic. Since $tst$ is also geodesic, we deduce that $\Geo_X(G)=L$. 
	
	It is well known that, for a hyperbolic group generated by $X$, $\Geo_X(G)$ is an automatic structure \cite{EpsteinEtAl}. It is obvious that the fellow traveler constant that works for $\Geo_X(G)$ also works for $L_0$, making it a geodesic automatic structure with uniqueness for $G$. 
\end{proof}

As an immediate corollary, we have that  $\Cay(G,X)$ is very close to being geodetic. Indeed, every element as a single geodesic (the one in $L_0$) except $tst=sts$.
\begin{corollary}\label{maincor}
	The Cayley graph $\Cay(G,X)$ is 2-geodetic but not geodetic.  The only element
with two geodesic representatives is $sts=tst$.
\end{corollary}

We now observe that $G$ is freely indecomposable. In particular, it is not plain, so  Shapiro's question and Elder--Piggott--Townsend's conjectures cannot both have a positive answer.

Moreover, in \cite[Theorem~3.1]{DromsServatiusServatius}, it is shown that a freely indecomposable group is either infinite cyclic or has a degree of connectivity of at least $2$. We conjecture that $2$-connectendess is an obstruction to geodeticity in the Cayley graph of a virtually free group.

\begin{proposition}
	\label{prop:two-connected}
	The group $G$ is freely indecomposable. As a consequence, $G$ is not plain and, for every generating set $X$ of $G$, the Cayley graph
	$\operatorname{Cay}(G,X)$ is $2$-connected.
\end{proposition}

\begin{proof}
	We first observe that $G$ is freely indecomposable. Suppose, to the
	contrary, that
	$
	G=H*K
	$
	is a nontrivial free product. Since $A\cong C_4$ and
	$B\cong D_3$ are finite, the Kurosh subgroup theorem implies that
	each is contained in a conjugate of one of the free factors.
	Since
	$$
	1\neq s\in A\cap B,
	$$
	Lemma~\ref{lem:free-factor-intersections} shows that these two conjugates
	of free factors must coincide. Hence $A$ and $B$ are contained in
	the same conjugate free factor. Since $A$ and $B$ generate $G$,
	this conjugate free factor is all of $G$, a contradiction.
	Therefore $G$ is freely indecomposable.
	
	If $G$ were plain, then, since it is freely indecomposable, it would
	have to be either finite or infinite cyclic. But $G$ is infinite and
	contains elements of finite order. Hence $G$ is not plain.
	
	Since $G$ is not infinite cyclic, it follows from
	\cite[Theorem~3.1]{DromsServatiusServatius} that no Cayley graph of
	$G$ has a cut vertex. Hence every Cayley graph of $G$ is
	$2$-connected.
\end{proof}

\begin{corollary}\label{cor:either}
	At least one of the following is false:
	\begin{enumerate}
		\item every geodetic group is plain;
		\item every hyperbolic $k$-geodetic group is geodetic.
	\end{enumerate}
	Thus Shapiro's conjecture and the first assertion of Conjecture C cannot
	both be true.
\end{corollary}

We conjecture that 2-connectedness of the Cayley graph of a virtually free group prevents geodeticity. The intuition is that trees have
unique geodesics, but they also have many cut vertices (in fact, if the tree is infinite, all vertices are cut vertices). If a graph
which looks like a tree from far away is also $2$-connected,
then there must be alternative routes around the cut vertices. In
an infinite vertex-transitive graph, one might expect these
alternative routes eventually to give two different shortest paths
with the same endpoints. 

\begin{question}
	Is there an infinite virtually free group $G$ admitting a finite generating set $X$ such that $\Cay(G,X)$ is geodetic and $2$-connected?
\end{question}

We also pose the more general graph-theoretical version of this question:
\begin{question}
	
	Is there an infinite, locally finite, $2$-connected,
	vertex-transitive geodetic graph which is quasi-isometric to a tree?
\end{question}

In \cite{EtgarLinial}, it is shown that $2$-connected geodetic graphs where all vertices have degree at least $3$ must be $3$-connected. Hence, a $2$-connected infinite geodetic Cayley graph must be $3$-connected (if it had a vertex of degree $2$ or less, then all vertices would have such a degree and the infinite graph would be a line, and so, not $2$-connected). 

So, to prove non-geodeticity of $G$, it would suffice to answer in the negative the following (weaker) version of the question:

\begin{question}
	Is there an infinite virtually free group $G$ admitting a finite generating set $X$ such that $\Cay(G,X)$ is geodetic and $3$-connected?
\end{question}

\section{Further work}

We are left with several questions. The main one concerns geodeticity.
\begin{question}
	Is the group $G$ geodetic?
\end{question}
A positive answer would solve Shapiro's question as the group is not plain. A negative answer would answer the conjecture from \cite{EPT} as we would have an example of a $2$-geodetic group that is not geodetic.

\begin{lemma}
	\label{lem:involution-odd-length}
	Let $H$ be a group and let $X$ be an inverse-closed generating set
	such that $\operatorname{Cay}(H,X)$ is geodetic. Then every
	nontrivial element of order $2$ of $H$ has odd $X$-length.
\end{lemma}

\begin{proof}
	Let $1\neq h\in H$ satisfy $h^2=1$, and let
	$
	w=x_1x_2\cdots x_n
	$
	be the unique geodesic word representing $h$. Since $h=h^{-1}$,
	the word
	$
	w^{-1}
	=
	x_n^{-1}\cdots x_2^{-1}x_1^{-1}
	$
	is another geodesic representative of $h$. By geodeticity,
	$w=w^{-1}$ as words. Hence
	$$
	x_i=x_{n+1-i}^{-1}
	$$
	for every $i$.
	
	If $n$ were even, then
	$
	x_{n/2}=x_{n/2+1}^{-1},
	$
	so the two middle letters of $w$ would be mutually inverse. The
	word $w$ would therefore not be geodesic, a contradiction. Hence
	$n$ is odd.
\end{proof}

\begin{proposition}
	\label{prop:no-geodetic-enlargement}
	Let
	$
	S=\{r,r^{-1},t,s\}.
	$
	Suppose that $X$ is an inverse-closed generating set containing $S$
	and that $\operatorname{Cay}(G,X)$ is geodetic. Then
	$
	G\setminus\{1\}\subseteq X.
	$
	In particular, no finite inverse-closed generating set containing
	$S$ gives a geodetic Cayley graph.
\end{proposition}

\begin{proof}
	Put
	$
	q=tst=sts.
	$
	The element $q$ has order $2$. Since $t,s\in X$, we have
	$
	|q|_X\leq3.
	$
	By Lemma~\ref{lem:involution-odd-length}, its $X$-length is odd.
	
	Suppose that $q\notin X$. Then $|q|_X\neq1$, and hence
	$
	|q|_X=3.
	$
	But
	$$
	tst
	\qquad\text{and}\qquad
	sts
	$$
	are two distinct words of length $3$ representing $q$, so both
	would be geodesic. This contradicts geodeticity. Therefore
	$
	q\in X
	$
	and $|q|_X=1$.
	
	Now put
	$
	a=ts.
	$
	Since $q=sts=tst$, we have
	$$
	a=ts=sq=qt.
	$$
	Thus $a$ has three distinct representatives of length $2$ over
	$X$. If $a\notin X$, then $|a|_X=2$, and all three displayed
	words would be geodesic. Therefore
	$
	ts\in X.
	$
	Since $X$ is inverse-closed,
	$$
	st=(ts)^{-1}\in X.
	$$
	
	Since
	$
	A=\{1,r,r^{-1},s\},
	$
	the assumption $S\subseteq X$ already gives
	$
	A\setminus\{1\}\subseteq X.
	$
	Moreover, we have proved that
	$
	t,s,q,ts,st\in X.
	$
	
	It follows that
	$$
	A\setminus\{1\}\subseteq X
	\qquad\text{and}\qquad
	B\setminus\{1\}\subseteq X.
	$$

	We now argue by induction on normal-form syllable length. Suppose
	that every nonidentity element whose normal form has fewer than
	$n$ syllables belongs to $X$, where $n\geq2$, and let
	$$
	g=u_1u_2\cdots u_ns^\epsilon,
	$$
	with $\epsilon \in \{0,1\}$ 
	be a normal form with $n$ syllables.
	
	Set
	$$
	x=u_1,
	\qquad
	y=u_2\cdots u_ns^\epsilon.
	$$
	The elements $x$ and $xs$ have normal-form syllable length one.
	Moreover, $y$ and $sy$ have normal-form syllable length $n-1$ (	notice that
	$
	sy=(su_2)u_3\cdots u_ns^\epsilon
	$
	is reduced because $s\in C$, so $su_2$ lies in the same factor as
	$u_2$, and $su_2\notin C$, since otherwise
	$
	u_2=s(su_2)\in C
	$).

	By the induction hypothesis,
	$
	x,  xs,  y,  sy \in X. 
	$
	Since $s^2=1$, then
	$$
	g=xy=(xs)(sy).
	$$
	These are two distinct words of length $2$ over $X$: their first
	letters are different because $x\neq xs$.
	
	If $g\notin X$, then $|g|_X=2$, so the two displayed words are
	distinct geodesic representatives of $g$, contradicting
	geodeticity. Hence
	$
	g\in X.
	$
	
	Therefore,
	$
	G\setminus\{1\}\subseteq X,
	$
	and, since $G$ is infinite, such an $X$ cannot be finite.
\end{proof}

Another question concerns finite convergent length-reducing rewriting
systems.  Madlener and Otto conjectured that the groups admitting such
presentations are precisely the plain groups \cite{MO}.  Since $G$ is
not plain, a positive answer to the following question would disprove
their conjecture. Elder and Piggott \cite[Theorem A]{EP} prove, in particular, that an
inverse-closed finite convergent length-reducing rewriting system gives
a geodetic Cayley graph.  Thus the generating set $S$ used above
cannot support such an inverse-closed system.

\begin{question}
	Does the group $G$ admit a presentation by a finite convergent length-reducing rewriting system?
\end{question}

\section*{Acknowledgements}
 The results in this paper were motivated in part by discussions with a commercial LLM, which suggested considering the group studied here as a possible counterexample to Conjecture~\ref{conjectureelderetal} and provided preliminary arguments concerning this example. The author verified the relevant claims independently and developed the more general results and obtained the proofs presented here. The author is solely responsible for the mathematical content of the paper. The author is grateful to Murray Elder for several helpful discussions which motivated further developments and alternative proofs, and to Kane Townsend for pointing out reference~\cite{EtgarLinial}.
The author  was supported by national funds through the Fundação para a Ciência e Tecnologia, FCT, under the project
UID/04674/2025.

 \end{document}